\documentclass[11pt]{article}

\usepackage[utf8]{inputenc}
\usepackage{amsfonts}
\usepackage{amsmath}
\usepackage{amsthm} 
\usepackage{booktabs}
\usepackage{color}
\usepackage{enumerate}
\usepackage{enumitem}
\usepackage{graphicx}
\usepackage{setspace}
\usepackage{subfigure}
\usepackage{url}

\newtheorem{theorem}{Theorem}[section]
\newtheorem{algorithm}[theorem]{Algorithm}
\newtheorem{procedure}[theorem]{Procedure}
\newtheorem{lemma}[theorem]{Lemma}
\newtheorem{definition}[theorem]{Definition}

\newtheorem{proposition}[theorem]{Proposition}
\newtheorem{assumption}[theorem]{Assumption}

\numberwithin{equation}{section}

\newcommand{\grad}{\nabla}                       
\newcommand{\inner}[2]{\langle#1,#2\rangle}      
\newcommand{\norm}[1]{\|#1\|}                    
\renewcommand{\Re}{\mathbb{R}}                   
\renewcommand{\S}{\mathbb{S}}                    
\newcommand{\T}{\top\hspace{-1pt}}               
\newcommand{\tr}{\mathrm{tr}}                    
\newcommand{\adjoint}[1]{\mathcal{A}^{*}({#1})}   

\begin{document}


\title{A revised sequential quadratic semidefinite programming
  method for nonlinear semidefinite optimization%
\thanks{This work was supported by the Grant-in-Aid for Early-Career
  Scientists (21K17709) and for Scientific Research (C) (19K11840)
  from Japan Society for the Promotion of Science.}
}

\author{Kosuke Okabe$^\dag$
  \and 
  Yuya Yamakawa$^\dag$
  \and
  Ellen H. Fukuda%
  \thanks{Department of Applied Mathematics and Physics, Graduate
    School of Informatics, Kyoto University, Kyoto \mbox{606--8501},
    Japan (\texttt{okabe@amp.i.kyoto-u.ac.jp}, \texttt{yuya@i.kyoto-u.ac.jp},
    \texttt{ellen@i.kyoto-u.ac.jp}).}
}

\maketitle


\begin{abstract}
  \noindent In 2020, Yamakawa and Okuno proposed a stabilized
  sequential quadratic semi\-definite programming (SQSDP) method for
  solving, in particular, degenerate nonlinear semidefinite
  optimization problems. The algorithm is shown to converge globally
  without a constraint qualification, and it has some nice properties,
  including the feasible subproblems, and their possible inexact
  computations.  In particular, the convergence was established for
  approximate-Karush-Kuhn-Tucker (AKKT) and trace-AKKT conditions,
  which are two sequential optimality conditions for the nonlinear
  conic contexts. However, recently, complementarity-AKKT (CAKKT)
  conditions were also consider, as an alternative to the previous
  mentioned ones, that is more practical. Since few methods are shown
  to converge to CAKKT points, at least in conic optimization, and
  to complete the study associated to the SQSDP, here we propose a
  revised version of the method, maintaining the good properties. We
  modify the previous algorithm, prove the global convergence in the
  sense of CAKKT, and show some preliminary numerical experiments.\\
  
\end{abstract}


\section{Introduction}
The following \emph{nonlinear semidefinite programming} (NSDP)
problem is considered:
\begin{equation}
  \label{eq:nsdp}
  \tag{NSDP}
  \begin{array}{ll}
    \underset{x \in \Re^n}{\mbox{minimize}} & f(x) \\ 
    \mbox{subject to} & g(x) = 0, \: X(x) \in \S^d_+,
  \end{array}
\end{equation}
where $f \colon \Re^n \to \Re$, $g \colon \Re^n \to \Re^m$, and $X
\colon \Re^n \to \S^d$ are twice continuously differentiable
functions, $\S^d$ is the linear space of all real symmetric matrices
of dimension $d \times d$, and $\S^d_+$ is the cone of all positive
semidefinite matrices in $\S^d$. The above problem extends some
well-known optimization problems, including the nonlinear programming
and the linear semidefinite programming (linear SDP).
\par
The research associated to NSDP is currently increasing, with many
possible applications, for instance in control theory~\cite{ANT03},
structural optimization~\cite{KT06,KS04} and finance~\cite{KKW03}. In
particular, methods like the interior point~\cite{Ja00,YY12,YYH12},
the augmented Lagrangian-type~\cite{ANT03,FAN01,FL18,SSZ08}, and the
sequential quadratic programming (SQP) type~\cite{CR04,FNA02,YO20} 
had been proposed in the last two decades (see~\cite{YY15} and references 
therein). Here, we focus in SQP-type methods for NSDP, which basically
consist in solving simpler quadratic SDP problems in each iteration, extending 
the classical SQP method for nonlinear programming. These methods are also called \emph{sequential quadratic semidefinite programming} (SQSDP).
One of the initial SQP-type 
methods for NSDP were proposed by Fares, Noll and Apkarian~\cite{FNA02} 
with a local convergent algorithm, and by Correa and Ram\'irez~\cite{CR04}, 
with a global convergent algorithm. Afterwards, some variants had been also 
proposed in the literature (see, for instance~\cite{KNKF05,ZC16,ZC20}).
\par
In particular, Yamakawa and Okuno~\cite{YO20} recently proposed a
\emph{stabilized SQSDP} method for NSDP, which is an extension of the 
stabilized SQP method for nonlinear programming proposed by Gill and 
Robinson~\cite{Gill2013}. The term ``stabilized'' is used because
stabilized subproblems are solved in each iteration in order to
handle degenerate problems. They proved that all accumulation points of the
sequence generated by the algorithm is either infeasible but stationary
for a feasibility 
problem, or satisfy some necessary optimality conditions. The conditions 
used in their work were actually proposed by Andreani, Haeser, and 
Viana~\cite{Andreani2018} for NSDP problems, and is called
\emph{approximate-Karush-Kuhn-Tucker} (AKKT) and \emph{trace-AKKT}
(TAKKT) conditions.
\par
These conditions are included in a class called \emph{sequential optimality 
conditions}~\cite{AHM11,AMS10}. They are shown to be necessary for optimality 
without any constraint qualification, which is a fact that differs from the 
classical KKT conditions. The sequential optimality conditions are essentially 
KKT variants, designed for building convergence theory of iterative algorithms.
The research related to these conditions in general conic programming context, 
in particular NSDP, are still ongoing.
Some kinds of sequential optimality conditions were proposed for NSDP so far, 
for example AKKT, TAKKT, and \emph{complementarity-AKKT} (CAKKT). 
In particular, CAKKT is a newer sequential optimality and was shown that 
CAKKT implies AKKT (TAKKT), and that CAKKT (or AKKT, TAKKT) with 
Robinson (or Mangasarian-Fromovitz) constraint qualification
imply KKT~\cite{AFHSS19,AFHSS21}.
\par
Moreover, some algorithms were shown to generate points satisfying
these sequential optimality conditions. However, in the conic context,
the few algorithms that were proved to generate, in particular, CAKKT
points are the augmented Lagrangian method~\cite{AFHSS21}, and the
primal-dual interior point method~\cite{AFHSS19}. To the best of authors'
knowledge, there is no SQP-type method producing CAKKT points.
Although Yamakawa and Okuno's~\cite{YO20} stabilized SQSDP method
can generate AKKT or TAKKT points, it is preferable that the method
obtains CAKKT ones, because as we mentioned, CAKKT is a sufficient 
condition under which AKKT (TAKKT) holds. Therefore, we refine the existing 
stabilized SQSDP method~\cite{YO20} so that it can find CAKKT points. Although
the modification in the algorithm is small, the proofs change considerably,
and it becomes necessary to assume the so-called generalized \L{}ojasiewicz
inequality to prove the global convergence.
\par
This paper is organized as follows. In Section~\ref{sec:preliminaries}, we
define some basic notations and review the sequential optimality
conditions for~\ref{eq:nsdp}. In Section~\ref{sec:method}, we describe
our proposed method, including the subproblem structure, the used
merit function, and ways to update the Lagrange multipliers and the
penalty parameters. We also prove the global convergence of the method
in Section~\ref{sec:convergence}, and present some numerical
experiments in Section~\ref{sec:experiments}. Finally, in
Section~\ref{sec:conclusion}, we show some final remarks.


\section{Preliminaries}
\label{sec:preliminaries}

Let us start with some notations that will be used throughout the
paper. The trace of a matrix $Z \in \S^d$ is denoted by $\tr(Z) :=
\sum_{i=1}^d Z_{ii}$. In addition, if $Y \in \S^d$, then the inner
product of $Y$ and $Z$ is written as $\inner{Y}{Z} := \tr (Y Z)$, and
the Frobenius norm of $Z$ is given by $\norm{Z}_F := \inner{Z}{Z}^{1/2}$.
We use the notation $Y \perp Z$ to denote $\inner{Y}{Z} = 0$.
Moreover, the eigenvalues of $Z \in \S^d$ are written as
$\lambda_1(Z) \leq \dots \leq \lambda_d(Z)$, in ascending order.
A positive definite (positive semidefinite) matrix $Z \in \S^d$ is
written as $Z \succ O$ ($Z \succeq O$).
We also define the following operator for all $Z \in \S^d$:
\[
\mathrm{svec}(Z) := (Z_{11}, \sqrt{2}Z_{21}, \dots, \sqrt{2}Z_{d1}, Z_{22},
\sqrt{2}Z_{32}, \dots, \sqrt{2}Z_{d2}, Z_{33}, \dots, Z_{dd})^{\T},
\]
where $\T$ means transpose. The inner product of vectors $x, y \in \Re^n$
is written as $\inner{x}{y} := \sum_{i = i}^{n}x_i y_i$ where $x_i$ and $y_i$
represent the $i$--th entry of $x$ and $y$ respectively, and the Euclidean
norm of $x$ is denoted by $\|x\| := \inner{x}{x}^{1/2}$.
For all $Y,Z \in \S^d$, we note that
\begin{equation}
  \label{eq:svec}
  \inner{Y}{Z} = \mbox{svec}(Y)^\T \mbox{svec}(Z), \quad \mbox{and} \quad
  \norm{Z}_F = \norm{\mbox{svec}(Z)}
\end{equation}
hold. For a given $\rho \colon \Re^n \rightarrow \Re$ and $x \in \Re^n$,
the gradient and the Hessian of $\rho$ at $x$ is written as
$\grad \rho(x)$ and $\grad^2 \rho(x)$, respectively.
If $\tilde{\rho} \colon \Re^n \times \S^d \to \Re$, then its gradient
at $(x,Z) \in \Re^n \times \S^d$ with respect to $x$ is denoted by
$\grad_x \tilde{\rho}(x,{Z})$. 

With the dimensions given by each context, we denote by $I$ the
identity matrix, and we use $e$ as the vector with all entries being
equal to one. Moreover, a diagonal matrix with diagonal entries $r_1,
\dots, r_n \in \Re$ is given by
\begin{equation*}
	\mathrm{diag}\left[r_1, \dots, r_n\right] := \left(
		\begin{array}{ccc}
			r_1 & & O \\
			& \ddots & \\
			O & & r_n
		\end{array}
	\right) \in \Re^{n \times n}.
\end{equation*}
Now, assuming that $Z \in \S^d$ can be diagonalized, we write its
decomposition as $Z = PDP^{\T}$, where $P$ is an orthogonal matrix and
$D$ is defined by $D = \mathrm{diag}\left[\lambda^{P}_{1}(Z), \dots,
  \lambda^{P}_{d}(Z)\right]$. Also, the projection of $Z$ onto
$\S^{d}_{+}$ is given as follows:
\begin{equation*}
  \left[Z\right]_{+} := P\,\mathrm{diag}\left[
    \left[\lambda^{P}_{1}(Z))\right]_{+},
    \dots, \left[\lambda^{P}_{d}(U)\right]_{+}\right]P^{\T},
\end{equation*}
where $\left[\,\cdot\,\right]_{+}\colon \Re \rightarrow \Re$ means
$\left[r\right]_{+} := \mathrm{max}\{0, r\}$. Furthermore, the
cardinality of a set $T$ is written as $\mathrm{card}(T)$.

Finally, let us define some notations related to
problem~\eqref{eq:nsdp}. Let $g := (g_1,\dots,g_m)$ with $g_i \colon
\Re^n \to \Re$ for all $i=1,\dots,m$. The transposed Jacobian matrix
of $g$ at $x$ is denoted by $\grad g(x) \in \Re^{n \times m}$, i.e.,
$\grad g(x) := \left[\grad g_{1}(x), \dots, \grad
  g_{m}(x)\right]$. The matrix $A_{j}(x) \in \S^{d}$ is defined as the
partial derivative $A_{j}(x) := \partial X(x) / \partial x_j $ for all
$j = 1, \dots, n$.  Also, the operator $\mathcal{A}(x)\colon \Re^n
\rightarrow \S^{d}$ and the adjoint operator $\mathcal{A}^{*}(x)\colon
\S^d \rightarrow \Re^n$ are given respectively~as:
\begin{align*}
  \mathcal{A}(x)u & := u_{1}A_{1}(x) + \dots + u_{n}A_{n}(x), \\
  \adjoint{x}U & := \left[\inner{A_{1}(x)}{U}, \dots, \inner{A_{n}(x)}{U}\right]^{\T}.
\end{align*}


\subsection{Optimality conditions}
\label{sec:opt_cond}

Here, we will review the concept of sequential optimality conditions
for NSDP, that was developed in~\cite{Andreani2018}. Let $L \colon \Re^{n} \times \Re^{m} \times
\S_{+}^{d} \rightarrow \Re$ be the Lagrangian function of~\eqref{eq:nsdp}:
\begin{equation}
	\label{eq:lagrange_function}
	L(x, y, Z) := f(x) - \inner{g(x)}{y} - \inner{X(x)}{Z},
\end{equation}
where $y \in \Re^{m}$ and $Z \in \S_{+}^{d}$ are the Lagrange multipliers associated to the equality and the conic constraints, respectively. From the definition~\eqref{eq:lagrange_function}, we observe that the gradient of $L$ with respect to $x$ is given by
\begin{equation}
	\label{eq:lagrange_function_grad}
	\grad_{x}L(x, y, Z) = \grad f(x) - \grad g(x)y - \adjoint{x}Z.
\end{equation}
%

\noindent Thus, we can define the KKT conditions of (\ref{eq:nsdp}) as follows.
\begin{definition}[KKT conditions]
  We say that $(x, y, Z) \in \Re \times \Re^{m} \times \S^{d}_{+}$ satisfies the KKT conditions
  for problem~\eqref{eq:nsdp} if
  \begin{eqnarray*}
		\begin{array}{c}
			\nabla_{x}L(x, y, Z) = 0,\ g(x) = 0, \\
			\inner{X(x)}{Z} = 0,\ X(x) \succeq O,\ Z \succeq O.
		\end{array}
	\end{eqnarray*}
\end{definition}
If $(x, y, Z)$ satisfies the above KKT conditions, then we call $x$ a KKT point, and $(y,Z)$ are the corresponding Lagrange multipliers. Moreover, as it is well known, a local optimal point needs to satisfy some constraint qualification in order to be a KKT point. As alternative optimality conditions, the sequential optimality have been studied in the last decade. They were first introduced for NLP and later developed for more general conic programming, which includes~NSDP \cite{AFHSS19,AHM11,Andreani2018}. Also, these conditions are known to be necessary for optimality without requiring a constraint qualification. In the following, we remark two sequential optimality conditions for NSDP problems, called AKKT and TAKKT.
\begin{definition} \cite[Definition~4]{Andreani2018}
  We say that $x \in \Re^n$ satisfies the AKKT conditions for problem~\eqref{eq:nsdp} if $g(x) = 0, X(x) \succeq O$ and there exist sequences $\left\{x_k\right\} \subset \Re^n, \left\{y_k\right\} \subset \Re^m$ and $\left\{Z_k\right\} \subset \S^{d}_{+}$ such that
  \begin{itemize}
  \item $\displaystyle{\lim_{k \to \infty}} x_k = x$,
  \item $\displaystyle{\lim_{k \to \infty}} \Big(\grad f (x_k)  - \grad g (x_k)y_k - \mathcal{A}^{*}(x_k) Z_k \Big) = 0$,
  \item $\lambda^{U}_{j}\big(X(x)\big) > 0 \Longrightarrow \mbox{there exists } k_{j} \in \mathbb{N}
    \mbox{ such that } \lambda^{U_k}_{j}(Z_k) = 0 \mbox{ for all } k \geq k_j$,\\
    where $U$ and $U_k$ are orthogonal matrices that satisfy $U_k \to U$ when $k \to \infty$, and
    \begin{eqnarray*}
      \begin{array}{c}
        X(x) = U \,\mathrm{diag}\Big[\lambda^{U}_{1}(X(x)), \dots, \lambda^{U}_  {d}(X(x))\Big]U^{\T}, \\[8pt]
        Z_k = U_{k} \,\mathrm{diag}\Big[\lambda^{U_k}_{1}(Z_k), \dots, \lambda^ {U_k}_{d}(Z_k)\Big]U_{k}^{\T}
      \end{array}
    \end{eqnarray*}
    for all $j = 1, \dots, d$.
  \end{itemize}
\end{definition}
\begin{definition} \cite[Definition~5]{Andreani2018}
  We say that $x \in \Re^n$ satisfies the TAKKT (trace-AKKT) conditions for problem~\eqref{eq:nsdp} if $g(x) = 0, X(x) \succeq O$ and there exist sequences $\left\{x_k\right\} \subset \Re^n, \left\{y_k\right\} \subset \Re^m$ and $\left\{Z_k\right\} \subset \S^{d}_{+}$ such that
\begin{itemize}
  \item $\displaystyle{\lim_{k \to \infty}} x_k = x$,
  \item $\displaystyle{\lim_{k \to \infty}} \Big(\grad f(x_k) - \grad g(x_k)y_k - \mathcal{A}^{*}(x_k)Z_k \Big) = 0$,
  \item $\displaystyle{\lim_{k \to \infty}} \inner{X(x_k)}{Z_k} = 0$.
\end{itemize}

\end{definition}
It was proved in~\cite[Theorems~2 and~5]{Andreani2018} that the local minimizers of (\ref{eq:nsdp}) always satisfy the AKKT and the TAKKT. Note also that, differently from the AKKT, the TAKKT avoids computation of eigenvalues. However, concerning the relation between these conditions, there are examples showing that the AKKT and the TAKKT conditions do not imply each other (see~\cite[Example 3]{Andreani2018} and~\cite[Example 3.1]{AFHSS19}). The following CAKKT was proposed as another sequential optimality condition, free of eigenvalue computations, more suitable for the conic context, but having also a clear relationship with both AKKT and TAKKT.
\begin{definition} \cite[Section~2]{AFHSS21}
  We say that $x \in \Re^n$ satisfies the CAKKT (complementarity-AKKT) conditions for problem~\eqref{eq:nsdp} if $g(x) = 0, X(x) \succeq O$ and there exist sequences $\left\{x_k\right\} \subset \Re^n, \left\{y_k\right\} \subset \Re^m$ and $\left\{Z_k\right\} \subset \S^{d}_{+}$ such that
\begin{itemize}
  \item $\displaystyle{\lim_{k \to \infty}} x_k = x$,
  \item $\displaystyle{\lim_{k \to \infty}} \Big(\grad f(x_k) - \grad g(x_k)y_k - \mathcal{A}^{*}(x_k)Z_k \Big) = 0$,
  \item $\displaystyle{\lim_{k \to \infty}} X(x_k) \circ Z_k = O$, \\
    where $\circ$ denotes the Jordan product, i.e., $A \circ B := \left(AB + BA\right)/2$ for all $A, B \in \S^{d}$.
\end{itemize}
\end{definition}
In \cite[Section 3]{AFHSS19} and~\cite[Theorem 2.3]{AFHSS21}, it was shown that the CAKKT implies both AKKT and TAKKT conditions. Furthermore, the CAKKT is equivalent to the KKT conditions when the Mangasarian-Fromovitz constraint qualification holds~\cite[Theorem 3.3]{AFHSS19}. In the next section, we will propose a method that generates these CAKKT sequences, by modifying the SQP-type method proposed in~\cite{YO20}.


\section{The proposed SQSDP method} \label{sec:method}
We describe a brief outline of our SQSDP method. The proposed SQSDP method is based on the SQP--type method developed in~\cite{YO20}, and mainly consists of three steps: solving a subproblem, updating the current point, and updating Lagrange multipliers and parameters. In the following, we provide explanation of each step.

\subsection{The subproblem of the proposed SQSDP method} \label{subsec:subproblem}
Let $k \in \mathbb{N}$ be the current iteration. For a given point $(x_{k}, y_{k}, Z_{k}) \in \Re^{n} \times \Re^{m} \times \S^{d}$, the proposed SQSDP method solves the following subproblem:
\begin{align}
	\label{eq:qsdp_subproblem}
	\begin{array}{cl}
		\underset{(\xi, \Sigma) \in \Re^{n} \times \S^{d}}{\mathrm{minimize}} &
		\inner{\grad f(x_{k}) - \grad g(x_{k})s_{k}}{\xi} + \frac{1}{2}\inner{M_{k} \xi}{\xi} + \frac{\sigma_{k}}{2}\norm{\Sigma}_{\mathrm{F}}^2 \\
		\mathrm{such\ that\ } &
		\mathcal{A}(x_{k})\xi + \sigma_{k} (\Sigma - T_{k}) \succeq O,
	\end{array}
\end{align}
where $\sigma_{k} > 0$ is a penalty parameter and $s_{k}$, $T_{k}$, and $M_{k}$ are defined as follows:
\begin{align*}
s_{k} := y_{k} - \frac{1}{\sigma_{k}}g(x_{k}), ~ T_{k} := Z_{k} - \frac{1}{\sigma_{k}}X(x_{k}), ~ M_{k} := H_{k} + \frac{1}{\sigma_{k}} \grad g(x_{k})\grad g(x_{k})^{\T},
\end{align*}
and $H_{k} \in \Re^{n \times n}$ is the Hessian $\nabla_{xx}^{2} L(x_{k}, y_{k}, Z_{k})$ of the Lagrangian function~\eqref{eq:lagrange_function} or its approximation. Problem~\eqref{eq:qsdp_subproblem} is derived from the following one, which is an extension of the existing stabilized subproblem proposed in~\cite{Wright1998}: 
\begin{align}
	\label{eq:qsdp_primitive}
	\begin{array}{cl}
		\underset{(\xi, \zeta, \Sigma) \in \Re^{n} \times \Re^{m} \times \S^{d}}{\mathrm{minimize}} &
		\inner{\grad f(x_{k})}{\xi} + \frac{1}{2}\inner{H_{k} \xi}{\xi} + \frac{\sigma_{k}}{2}\norm{\zeta}^2 + \frac{\sigma_{k}}{2}\norm{\Sigma}_{\mathrm{F}}^{2} \\
		\mathrm{subject\ to\ } & g(x_{k}) + \grad g(x_{k})^{\T}\xi + \sigma_{k} (\zeta - y_{k}) = 0, \\
		& X(x) + \mathcal{A}(x_{k})\xi + \sigma_{k} (\Sigma - Z_{k}) \succeq O.
	\end{array}
\end{align}
Problem~\eqref{eq:qsdp_subproblem} is obtained by eliminating the variable $\zeta$ in~\eqref{eq:qsdp_primitive} via $\zeta = y_{k} - \frac{1}{\sigma_{k}} ( g(x_{k}) + \nabla g(x_{k})^{\top} \xi_{k} )$ and has various useful properties below~\cite{YO20}:
\begin{itemize}
\item[(i)] It always has a strictly feasible point $(\xi, \Sigma) = (0, I+T_{k})$;
\item[(ii)] if $M_{k} \succ O$, then it has a unique global optimum.
\end{itemize}
Item~(i) implies that \eqref{eq:qsdp_subproblem} is solvable even if the current point $(x_{k}, y_{k}, Z_{k})$ is not sufficiently close to the KKT point of~\eqref{eq:nsdp} and satisfies Slater's constraint qualification. Moreover, these facts and item~(ii) ensure that if the approximate Hessian $H_{k}$ is designed so that $M_{k} \succ O$ at each iteration, then the proposed SQSDP can obtain the KKT point of \eqref{eq:qsdp_subproblem}. Therefore, we suppose that $M_{k} \succ O$ in the subsequent discussion. After obtaining the unique optimum $(\xi_{k}, \Sigma_{k})$ of \eqref{eq:qsdp_subproblem}, we set a search direction $p_{k}$ and trial Lagrange multipliers $\overline{y}_{k+1}$ and $\overline{Z}_{k+1}$ as
\begin{align} \label{def:overline}
p_{k} := \xi_{k}, \quad \overline{y}_{k+1} := y_{k} - \frac{1}{\sigma_{k}} \left( g(x_{k}) + \nabla g(x_{k})^{\top} \xi_{k} \right), \quad \overline{Z}_{k+1} := \Sigma_{k},
\end{align}
respectively.

\subsection{Updating the current iterate} \label{subsec:merit_function}
After computing the search direction $p_{k}$, we consider updating $x_{k}$ along $\xi_{k}$. To decide the step size which indicates how far we update $x_{k}$ along $\xi_{k}$, we introduce the following merit function $F \colon \Re^{n} \rightarrow \Re$:
\begin{align}
  \label{eq:merit_function}
  F(x; \sigma, y, Z) := f(x) + \frac{1}{2\sigma}\norm{\sigma y - g(x)}^2
  + \frac{1}{2\sigma}\norm{\left[\sigma Z - X(x)\right]_{+}}_{\mathrm{F}}^2,
\end{align}
where only $x$ is the variable.  From~\cite[Lemma 5]{Andreani2018},
the merit function $F$ is differentiable on~$\Re^n$ and its gradient
at $x$ is given by
\begin{align}
  \label{eq:grad_merit_function}
  \grad F(x; \sigma, y, Z)
  = \grad f(x) - \grad g(x) \left(y -  \frac{1}{\sigma} g(x) \right) - \adjoint{x} \left[Z - \frac{1}{\sigma} X(x) \right]_{+}.
\end{align}
Furthermore, the merit function $F$ has the following nice property.
\begin{proposition}{\rm \cite[Proposition~2]{YO20}} \label{pro:descent_direction}
Assume that $M_{k} \succ O$. Then, the unique optimum $(\xi_{k}, \Sigma_{k})$ of \eqref{eq:qsdp_subproblem} satisfies $\langle \nabla F(x_{k}; \sigma_{k}, y_{k}, Z_{k}), \xi_{k} \rangle \leq - \langle M_{k} \xi_{k}, \xi_{k} \rangle - \sigma_{k} \Vert \Sigma_{k} - Z_{k} \Vert_{{\rm F}}^{2}$. Moreover, $\nabla F(x_{k}; \sigma_{k}, y_{k}, Z_{k}) = 0$ if and only if $(\xi_{k}, \Sigma_{k}) = (0, [T_{k}]_{+})$.
\end{proposition}

Proposition~\ref{pro:descent_direction} guarantees that $p_{k} = \xi_{k}$ is a descent direction of $F$. By applying a backtracking line search with $F$, we determine the step size $\alpha_{k}$, namely, we set $\alpha_{k} := \beta^{\ell_{k}}$, where $\beta \in (0, 1)$ is a positive constant and $\ell_{k} \geq 0$ is the smallest nonnegative integer such~that
\begin{align} \label{stepsize:Armijo}
F(x_{k} + \beta^{\ell_{k}} p_{k}; \sigma_{k}, y_{k}, Z_{k}) \leq F(x_{k}; \sigma_{k}, y_{k}, Z_{k}) + \tau \beta^{\ell_{k}} \Delta_{k},
\end{align}
with $\Delta_{k} := \max \left\{ \inner{\grad F(x_{k}; \sigma_{k}, y_{k}, Z_{k})}{p_{k}}, -\omega \norm{p_{k}}^2 \right\}$, $\omega \in (0, 1)$, and $\tau \in (0, 1)$. Using the step size $\alpha_{k}$, the current point $x_{k}$ is updated by $x_{k+1} := x_{k} + \alpha_{k} p_{k}$.

\subsection{Updating Lagrange multipliers}
\label{subsec:update_lagrange_multiplier}
This section describes a procedure to update Lagrange multipliers. Although the ordinary SQP--type methods would immediately set the new Lagrange multiplier pair $(y_{k+1}, Z_{k+1})$ as $(\overline{y}_{k+1},\overline{Z}_{k+1})$, the proposed SQSDP method does not update them like the ordinary methods and first check whether the triplet $(x_{k+1}, \overline{y}_{k}, \overline{Z}_{k+1})$ is approaching the KKT or CAKKT point. To this end, we define two functions to measure the distance between a given point $(x,y,Z)$ and the KKT point:
\begin{align}
  \label{eq:dist}
  \Phi(x, y, Z) := r_{\mathrm{V}}(x) + \kappa r_{\mathrm{O}}(x, y, Z), \quad
  \Psi(x, y, Z) := \kappa r_{\mathrm{V}}(x) + r_{\mathrm{O}}(x, y, Z),
\end{align}
where $\kappa \in (0, 1)$ and
\begin{align}
  \label{eq:rv_ro}
  r_{\mathrm{V}}(x) := \norm{g(x)} + \left[\lambda_{\mathrm{max}}(-X(x))\right]_{+}, \quad
  r_{\mathrm{O}}(x, y, Z) := \norm{\grad_{x}L(x, y, Z)} + \norm{X(x)Z}_{\mathrm{F}}.
\end{align}
Note that $r_{\mathrm{V}}(x) =  r_{\mathrm{O}}(x, y, Z) = 0$ if and only if $(x, y, Z)$ is the KKT point of \eqref{eq:nsdp}. Moreover, we utilize the function $\Vert \nabla F( \, \cdot \, ; \sigma_{k}, y_{k}, Z_{k}) \Vert$ to measure the distance between a given point $x$ and the CAKKT point. By combining these concepts, we provide the following procedure to update the Lagrange multipliers. This procedure\footnote{The V, O, M and F-iterations mean violation, optimality, merit (function), and failure, respectively.} was first proposed by Gill and Robinson~\cite{Gill2013} and extended for NSDP by Yamakawa and Okuno~\cite{YO20}. We point out that the main difference between this method and our proposal is in the definition of the function $r_O$ (the second definition in~\eqref{eq:rv_ro}), and consequently, in the functions~$\Phi$ and~$\Psi$ given in~\eqref{eq:dist}.
\begin{procedure} \label{algorithm:procedure}
\
\begin{description}

\item[Step 1.]
(V-iterate) If $\Phi(x_{k+1}, \overline{y}_{k+1}, \overline{Z}_{k+1}) \leq \frac{1}{2}\phi_{k}$, then set
\begin{align*}
\phi_{k+1} := \frac{1}{2}\phi_{k},\ \psi_{k+1} := \psi_{k},\ \gamma_{k+1} := \gamma_{k},\ y_{k+1} := \overline{y}_{k+1},\ Z_{k+1} := \overline{Z}_{k+1},
\end{align*}
and end the procedure. Otherwise, go to Step $2$.

\item[Step 2.]
(O-iterate) If $\Psi(x_{k+1}, \overline{y}_{k+1}, \overline{Z}_{k+1}) \leq \frac{1}{2}\psi_{k}$, set
\begin{align*}
\phi_{k+1} := \phi_{k},\ \psi_{k+1} := \frac{1}{2}\psi_{k},\ \gamma_{k+1} := \gamma_{k},\ y_{k+1} := \overline{y}_{k+1},\ Z_{k+1} := \overline{Z}_{k+1},
\end{align*}
and end the procedure. Otherwise, go to Step $3$.

\item[Step 3.]
(M-iterate) If $\norm{\grad F(x_{k+1}; \sigma_k, y_k, Z_k)} \leq \gamma_k$, set
\begin{align*}
\begin{array}{c}
  \phi_{k+1} := \phi_{k}, \quad
  \psi_{k+1} := \psi_{k}, \quad
  \gamma_{k+1} := \frac{1}{2}\gamma_{k}, \\[5pt]
  y_{k+1} := \Pi_{C}\big( y_{k} - \frac{1}{\sigma_{k}}g(x_{k+1}) \big),
  \quad Z_{k+1} := \Pi_{D} \big([Z_{k} - \frac{1}{\sigma_{k}}X(x_{k+1})]_{+} \big),
\end{array}
\end{align*}
and end the procedure. Here, $\Pi_C$ and $\Pi_D$ are projections onto
the sets
\[
C := \left\{y \in \Re^{m} | \, -y_{\max}e \leq y \leq
y_{\max} e \right\}
\quad \mbox{and} \quad D := \left\{ Z \in \S^{d} | \, O \preceq Z
\preceq z_{\max}I \right\},
\]
respectively, and $y_{\max} > 0$, 
$z_{\max} > 0$ are constants. Otherwise, go to Step $4$.

\item[Step 4.]
(F-iterate) Set $\phi_{k+1} := \phi_{k},\ \psi_{k+1} := \psi_{k},\ \gamma_{k+1} := \gamma_{k},\ y_{k+1} := y_k,\ Z_{k+1} := Z_k$.
\end{description}

\end{procedure}

If the conditions $\Phi(x_{k+1}, \overline{y}_{k+1}, \overline{Z}_{k+1}) \leq \frac{1}{2}\phi_{k}$ and $\Psi(x_{k+1}, \overline{y}_{k+1}, \overline{Z}_{k+1}) \leq \frac{1}{2} \psi_{k}$ in Steps~1 and 2 are satisfied, then we consider that $(x_{k+1}, \overline{y}_{k}, \overline{Z}_{k+1})$ is approaching the KKT point. Therefore, we adopt the trial Lagrange multiplier pair $(\overline{y}_{k+1}, \overline{Z}_{k+1})$ as the new Lagrange multiplier pair $(y_{k+1}, Z_{k+1})$ because the pair has a nice tendency. If $\norm{\grad F(x_{k+1}; \sigma_k, y_k, Z_k)} \leq \gamma_k$ in Step~3 is satisfied, then we consider that $x_{k+1}$ is approaching the CAKKT point. In this case, we adopt the updating rule used in the augmented Lagrangian method~\cite{Andreani2018} because the function $F$, which is regarded as the augmented Lagrange function, is minimizing. If all the conditions in Steps 1, 2, and 3 are not satisfied, we do not update the Lagrange multipliers in this iteration.

\subsection{Updating the penalty parameter} \label{subsec:update_penalty_parameter}
This paper considers that \eqref{eq:nsdp} does not satisfy any constraint qualification, and hence it is possible that there is no Lagrange multiplier pair satisfying the KKT conditions. Therefore, we need to design the proposed method so that it can find a CAKKT point. This is done by minimizing the merit function $F$, which is called the augmented Lagrange function, as shown in the next section. Since the case where $\norm{\grad F(x_{k+1}; \sigma_k, y_k, Z_k)} \leq \gamma_k$, seen in Step~3 in Procedure~\ref{algorithm:procedure}, corresponds to the step which minimizes the augmented Lagrangian function, it is reasonable to update the penalty parameter $\sigma_{k}$ based on the manner used in the augmented Lagrangian method as follows:
\begin{align} \label{rule:sigma}
  \sigma_{k+1} :=
  \left\{
  \begin{array}{ll}
    \mathrm{min}\big\{\frac{1}{2}\sigma_k, r(x_{k+1}, y_{k+1}, Z_{k+1})^{\frac{3}{2}}\big\},
    &\mathrm{if}\ \norm{\grad F(x_{k+1}; \sigma_k, y_k, Z_k)} \leq \gamma_k, \\
    \sigma_k, & \mathrm{otherwise},
  \end{array}
  \right.
\end{align}
where
\[
r(x, y, Z) := r_{\mathrm{V}}(x) + r_{\mathrm{O}}(x, y, Z)
\]
and $r_V, r_O$ are defined in~\eqref{eq:rv_ro}. Note that the term
$r(x_{k+1}, y_{k+1}, Z_{k+1})^{\frac{3}{2}}$ has an effect to achieve
fast local convergence and is also utilized in \cite{Gill2013}.

\subsection{Proposed algorithm} \label{subsec:proposed_method}
Summarizing the above discussion, we propose the following method for solving (\ref{eq:nsdp}).
\begin{algorithm} \label{algorithm:proposed}
\
\begin{description}
\item[Step~0.]
Set constants $\tau \in (0, 1), ~ \omega \in (0, 1), ~ \beta \in (0, 1), ~ \kappa \in (0, 1), ~ y_{\max} > 0, ~ z_{\max} > 0$, $k_{\max} \in \mathbb{N}$, and $\varepsilon > 0$. Choose an initial point $(x_0, y_0, Z_0)$ and parameters $\phi_0 > 0, ~ \psi_0 > 0, ~ \gamma_0 > 0, ~ \sigma_0 > 0, ~ k:= 0, ~ \overline{y}_0 := y_0$, and $\overline{Z}_0 := Z_0$. Go to Step~{\rm 1}.

\item[Step~1.]
If $r(x_k, y_k, Z_k) \leq \varepsilon$, $\gamma_k \leq \varepsilon$, or $k = k_{\max}$, then stop. Otherwise, go to Step~{\rm 2}.

\item[Step~2.]
If $\norm{\grad F(x_k; \sigma_k, y_k, Z_k)} = 0$, set
\begin{align*}
  x_{k+1} := x_k, \quad
  \overline{y}_{k+1} := y_k - \frac{1}{\sigma_k}g(x_{k+1}), \quad
  \overline{Z}_{k+1} := \left[Z_k - \frac{1}{\sigma_k}X(x_{k+1})\right]_{+}
\end{align*}
and go to Step~{\rm 5}. Otherwise, go to Step~{\rm 3}.
		
\item[Step~3.]
Choose $H_k \succ O$ and find the global optimum $(\xi_{k}, \Sigma_{k})$ by solving \eqref{eq:qsdp_subproblem}. Set
\begin{align*}
  p_k := \xi_{k}, \quad
  \overline{y}_{k+1} := y_{k} - \frac{1}{\sigma_{k}}
  \left( g(x_{k}) + \grad g(x_{k})^{\T}\xi_{k} \right), \quad
  \overline{Z}_{k+1} := \Sigma_{k}
\end{align*}
and go to Step~{\rm 4}.
		
\item[Step~4.]
Compute the smallest nonnegative integer $\ell_k$ with \eqref{stepsize:Armijo}. Set $x_{k+1} := x_k + \beta^{\ell_k}p_k$ and go to Step~{\rm 5}.

\item[Step~5.]
Compute $y_{k+1},\ Z_{k+1},\ \phi_{k+1},\ \psi_{k+1}$ and $\gamma_{k+1}$ with Procedure~{\rm \ref{algorithm:procedure}} and go to Step~{\rm 6}.
		
\item[Step~6.]
Update $\sigma_{k}$ using \eqref{rule:sigma} and go to Step~{\rm 7}.

\item[Step~7.]
Set $k := k+1$ and go to Step~{\rm 1}.
\end{description}
\end{algorithm}

Note that if $\norm{\grad F(x_k; \sigma_k, y_k, Z_k)} = 0$ holds in Step $2$, then Proposition~\ref{pro:descent_direction} indicates that the unique global minimizer of \eqref{eq:qsdp_subproblem} is $(\xi_{k}, \Sigma_{k}) = (0, [ T_{k} ]_{+})$. Thus, in this case, we do not need to solve the subproblem~\eqref{eq:qsdp_subproblem} and we go to Step~5 immediately.


\section{Global convergence}
\label{sec:convergence}

In this section, we discuss convergence properties of
Algorithm~\ref{algorithm:proposed}, assuming that an infinite sequence
of points are generated. For that, we suppose that the following
assumption holds.
\begin{assumption}
  Let $\{x_k\}$ be a sequence generated by Algorithm~\ref{algorithm:proposed}
  for problem~\eqref{eq:nsdp}. Then, we assume the following assertions:
	\label{assumption:function}
	\begin{description}[leftmargin=22pt]
		\item[(a)] The functions $f,\ g$ and $X$ are twice continuously differentiable.
		\item[(b)] There exists a compact set containing $\{x_k\}$.
		\item[(c)] For all $k$, there exist constants $\nu_1$
                  and $\nu_2$ satisfying
		\begin{eqnarray*}
		  \nu_1 \leq \lambda_{\mathrm{min}}
                  \left(H_k + \frac{1}{\sigma_k}\grad g(x_k)\grad g(x_k)^{\T}\right)
                  \quad \mbox{and} \quad
		  \lambda_{\mathrm{max}}(H_k) \leq \nu_2,
		\end{eqnarray*}
		where $\lambda_{\mathrm{min}}(A)$ and $\lambda_{\mathrm{max}}(A)$ denote, respectively, the minimum and the maximum eigenvalues of $A \in \S^{d}$.
	      \item[(d)] Let $x^{*}$ be an accumulation point of $\left\{x_k\right\}$.
                Then, there exist $\delta > 0$ and a continuous function $\varphi \colon B(x^{*}, \delta)\rightarrow\Re$ satisfying $\lim_{x \to x^{*}}\varphi(x) = 0$ and
		\begin{eqnarray*}
			|P(x)-P(x^{*})| \leq \varphi(x)\norm{\grad P(x)},
		\end{eqnarray*}
		where $B(x^{*}, \delta)$ is the Euclidean ball with radius $\delta$ around $x^{*}$,
                and $P$ is a feasibility measure defined by
                \begin{equation}
                  \label{eq:feasib}
                  P(x) := \frac{1}{2}\norm{g(x)}^2
                  + \frac{1}{2} \norm{\left[-X(x)\right]_{+}}_{\mathrm{F}}^2.
                \end{equation}
	\end{description}
\end{assumption}
In the above assumption, we note that (a) and (b) are standard, and
also used in~\cite{YO20}. The assumption (c) is also used in the same
work, and it holds, for instance, when $H_k$ is positive definite and
bounded~\cite{YO20}. Moreover, (d) is a weak assumption on the
smoothness of~$g(\cdot)$ and $X(\cdot)$, which is defined
in~\cite{AMS10} and known as a generalized \L{}ojasiewicz inequality.

Before showing the main convergence result, let us consider the
following partitions for the algorithm's iterations:
\begin{align*}
  \mathcal{K}_{VO} := &
  \left\{k \mid \text{V-iterate or O-iterate is executed in\ }
  k\text{-th iteration of Procedure\ } \ref{algorithm:procedure} \right\}, \\
  \mathcal{K}_M := & \left\{k \mid \text{M-iterate is executed in\ }
  k\text{-th iteration of Procedure\ } \ref{algorithm:procedure} \right\}, \\
  \mathcal{K}_F := & \left\{k \mid \text{F-iterate is executed in\ }
  k\text{-th iteration of Procedure\ } \ref{algorithm:procedure} \right\}.
\end{align*}
We now give two lemmas, associated with the above set of iterations.
\begin{lemma}
  \label{lemma:cardinality}
  Suppose that Assumption \ref{assumption:function} holds.
  Then, we have:
  \begin{itemize}
  \item[(i)] If $\mathrm{card}(\mathcal{K}_{VO}) = \infty$,
    then $\phi_k \to 0$ or $\psi_k \to 0$ when $k \to \infty$.
  \item[(ii)] If $\mathrm{card}(\mathcal{K}_{VO}) < \infty$,
    then $\left\{Z_k\right\}$ is bounded.
  \item[(iii)] If $\mathrm{card}(\mathcal{K}_{VO}) < \infty$
    and $\mathrm{card}(\mathcal{K}_M) = \infty$,
    then $\sigma_k \to 0$ and $\gamma_k \to 0$ when $k \to \infty$.
  \item[(iv)] The situation $\mathrm{card}(\mathcal{K}_{VO}) < \infty$,
    $\mathrm{card}(\mathcal{K}_M) < \infty$,
    and $\mathrm{card}(\mathcal{K}_F) = \infty$ never occurs.
  \end{itemize}
\end{lemma}

\begin{proof}
  See~\cite[Lemma~4 and Theorem~3]{YO20}.
\end{proof}

\begin{lemma}
  \label{lemma:convergence}
  Let $\left\{x_k\right\}$ be a sequence generated by
  Algorithm~\ref{algorithm:proposed}, and suppose that Assumption
  \ref{assumption:function} holds. Assume that
  $\mathrm{card}(\mathcal{K}_{VO}) < \infty$ and define
  $\mathcal{K}_M' := \left\{k \in \mathbb{N} \mid k-1 \in
  \mathcal{K}_M\right\}$. Moreover, let $\mathcal{M} \subset
  \mathcal{K}_M'$ such that $x_k \to x^{*}$ when $\mathcal{M} \ni k
  \to \infty$, where $x^{*}$ is a feasible point of~\eqref{eq:nsdp}.
  Then the following statements hold:
  \begin{itemize}
  \item[(i)] $\displaystyle{(1/\sigma_{k-1})\lambda_{i}\big(
    \sigma_{k-1}Z_{k-1} - X(x_k)\big)
    \big[\lambda_{i}(\sigma_{k-1}Z_{k-1} - X(x_k))\big]_{+} \to 0}$
    when $\mathcal{M} \ni k \to \infty$ for all $i = 1, \dots, m$.
  \item[(ii)] $\left\|\grad P(x_k)/\sigma_{k-1}\right\|$ is bounded
    when $k \in \mathcal{M}$.
  \end{itemize}
\end{lemma}

\begin{proof}
  See Appendix A.
\end{proof}

Under Assumption \ref{assumption:function}, and supposing also that an
infinite sequence of points are generated by
Algorithm~\ref{algorithm:proposed}, we obtain the following
convergence result. Basically, any accumulation point of the sequence
generated by the method is either a CAKKT point, or it is infeasible
but stationary for the feasibility measure~$P$.

\begin{theorem}
  Let $\left\{x_k\right\}$ be a sequence generated by Algorithm~\ref{algorithm:proposed}, and
  suppose that Assumption \ref{assumption:function} holds. Then, any accumulation point of $\{x_k\}$ satisfies one of the following:
  \begin{itemize}
  \item[(i)] It is a CAKKT point of (\ref{eq:nsdp});
  \item[(ii)] It is an infeasible point of (\ref{eq:nsdp}),
    but a stationary point of the feasibility measure~$P$, defined in~\eqref{eq:feasib}.
  \end{itemize}
\end{theorem}

\begin{proof}
  We consider two cases:
  (a)\ $\mathrm{card}(\mathcal{K}_{VO}) = \infty$, and
  (b)\ $\mathrm{card}(\mathcal{K}_{VO}) < \infty$. \\

  \noindent \textbf{Case (a)} Assume that
  $\mathrm{card}(\mathcal{K}_{VO}) = \infty$.  We show that
  statement ($i$) holds in this case.  Let $\mathcal{K}_{VO}' :=
  \left\{k \in \mathbb{N} \mid k-1 \in \mathcal{K}_{VO}\right\}$. It
  is clear that $\mathrm{card}(\mathcal{K}_{VO}') = \infty$ because
  $\mathrm{card}(\mathcal{K}_{VO}) = \infty$. Using this fact and
  Assumption \ref{assumption:function} (b), $\left\{x_k\right\}_{k \in
    \mathcal{K}_{KO}'}$ has an accumulation point, say $x^{*}$.  Then,
  there exists $\mathcal{J} \subset \mathcal{K}_{KO}'$ such that $x_k
  \to x^{*}$ when $\mathcal{J} \ni k \to
  \infty$. From~\eqref{eq:dist}, Lemma~\ref{lemma:cardinality}~($i$),
  and the steps in Procedure~\ref{algorithm:procedure}, it follows
  that $r_{\mathrm{V}}(x_k) \to 0$ and $r_{\mathrm{O}}(x_k, y_k, Z_k)
  \to 0$ when $\mathcal{J} \ni k \to \infty$.  Therefore, from the
  definition of $r_{\mathrm{V}}$ and $r_{\mathrm{O}}$ given
  in~\eqref{eq:rv_ro}, we conclude that
  \begin{align*}
    \underset{\mathcal{J} \ni k \to \infty}{\mathrm{lim}}
    \Big(\grad f(x_k) - \grad g(x_k)y_k - \adjoint{x_k}Z_k\Big) = 0, \\
    \underset{\mathcal{J} \ni k \to \infty}{\mathrm{lim}} \norm{X(x_k)Z_{k}}_{\mathrm{F}}
    = \underset{\mathcal{J} \ni k \to \infty}{\mathrm{lim}} \norm{Z_{k}X(x_k)}_{\mathrm{F}} = 0, \\
    \underset{\mathcal{J} \ni k \to \infty}{\mathrm{lim}} g(x_k) = g(x^{*}) = 0, \\
    \underset{\mathcal{J} \ni k \to \infty}{\mathrm{lim}}
    \left[\lambda_{\mathrm{max}}(-X(x_k))\right]_{+}
    = \left[\lambda_{\mathrm{max}}(-X(x^{*}))\right]_{+} = 0,
  \end{align*}
  which show that $x^{*}$ is a CAKKT point of (\ref{eq:nsdp}). \\

  \noindent \textbf{Case (b)} Assume that $\mathrm{card}(\mathcal{K}_{VO}) < \infty$. 
  To begin with, we show that the gradient of the Lagrange function
  converge to $0$ in some subset of $\mathcal{K}_M$. Recalling that an
  infinite sequence of iterates are generated, it follows from Lemma
  \ref{lemma:cardinality} ($iv$) that $\mathrm{card}(\mathcal{K}_M) =
  \infty$. This shows that $\mathrm{card}(\mathcal{K}_M') = \infty$,
  where $\mathcal{K}_M' := \left\{k \in \mathbb{N} \mid k-1 \in
  \mathcal{K}_M\right\}$. Once again by Assumption
  \ref{assumption:function} (b), $\left\{x_k\right\}_{k \in
    \mathcal{K}_M'}$ has at least one accumulation point. Let $x^{*}$
  be such a point. Thus, there exists $\mathcal{J} \subset
  \mathcal{K}_M'$ such that $x_k \to x^{*}$ when $\mathcal{J} \ni k
  \to \infty$. Moreover, from Lemma \ref{lemma:cardinality}
  $(ii)$--$(iii)$, we have the boundedness of $\left\{Z_k\right\}$,
  and that $\sigma_k \to 0$ and $\gamma_k \to 0$ when $k \to \infty$.
  Defining for simplicity the terms below,
  \[
  \tilde{y}_k := y_{k-1} - \frac{1}{\sigma_{k-1}}g(x_k), \quad
  \tilde{Z}_k := \left[Z_{k-1} - \frac{1}{\sigma_{k-1}}X(x_k)\right]_{+},
  \]
  and recalling~\eqref{eq:grad_merit_function}, we obtain
  \begin{eqnarray*}
    \norm{\grad f(x_k) - \grad g(x_k)\tilde{y}_k - \adjoint{x_k}\tilde{Z}_k}
    = \norm{\grad F(x_{k}; \sigma_{k-1}, y_{k-1}, Z_{k-1})} \leq \gamma_{k-1},
  \end{eqnarray*}
  when $k \in \mathcal{J}$. These facts and Lemma \ref{lemma:cardinality} $(iii)$ yield 
  \begin{eqnarray}
    \label{eq:lagrange_gradient_converge_to_zero}
    \lim_{\mathcal{J} \ni k \to \infty} \Big(\grad f(x_k) - \grad g(x_k)\tilde{y}_k
    - \adjoint{x_k}\tilde{Z}_k\Big) = 0.
  \end{eqnarray}
  %
  Now, we consider two cases: (b1) when $x^{*}$ is feasible for (\ref{eq:nsdp}),
  and (b2) when $x^{*}$ is an infeasible point of (\ref{eq:nsdp}).  We
  will show that ($i$) holds in the first case, and ($ii$) is satisfied
  in the latter case.\\

  \noindent \textbf{Case (b1)} Let us show that situation ($i$) holds
  when $x^{*}$ is feasible for (\ref{eq:nsdp}). From the definition of
  CAKKT, as well as~(\ref{eq:lagrange_gradient_converge_to_zero}), it is
  sufficient to show that $X(x_k) \circ \tilde{Z}_k \to O$ when $k \to
  \infty$. Moreover, it is sufficient to prove that $X(x_k)\tilde{Z}_k
  \to O$ when $k \to \infty$ since $\norm{X(x_k)
    \tilde{Z}_k}_{\mathrm{F}} = \norm{\tilde{Z}_k
    X(x_k)}_{\mathrm{F}}$.

  Consider the diagonalization $\sigma_{k-1}Z_{k-1} - X(x_k) =
  S_{k}D_{k}S_{k}^{\T}$, where $S_k$ is an orthogonal matrix. Hence, simple
  calculations show that
  \begin{align*}
    X(x_k)\tilde{Z}_k
    = & \; (\sigma_{k-1}Z_{k-1} - S_{k}D_{k}S_{k}^{\T})
    \left[Z_{k-1} - \frac{1}{\sigma_{k-1}}X(x_k)\right]_{+} \\
    = & \; \sigma_{k-1}Z_{k-1}
    \left[Z_{k-1} - \frac{1}{\sigma_{k-1}}X(x_k)\right]_{+} - S_{k}D_{k}S_{k}^{\T}
    \left[Z_{k-1} - \frac{1}{\sigma_{k-1}}X(x_k)\right]_{+} \\
    = & \; \sigma_{k-1}Z_{k-1}
    \left[Z_{k-1} - \frac{1}{\sigma_{k-1}}X(x_k)\right]_{+} - S_{k}D_{k}S_{k}^{\T}
    \left[\frac{1}{\sigma_{k-1}}S_{k}D_{k}S_{k}^{\T}\right]_{+} \\
    = & \; \sigma_{k-1}Z_{k-1}
    \left[Z_{k-1} - \frac{1}{\sigma_{k-1}}X(x_k)\right]_{+} - \frac{1}{\sigma_{k-1}}S_{k}D_{k}
    \left[D_{k}\right]_{+}S_{k}^{\T}.
  \end{align*}
  We now prove that
  $\sigma_{k-1}Z_{k-1}\left[Z_{k-1} - X(x_k)/\sigma_{k-1}\right]_{+}$
  and $(1/\sigma_{k-1})S_{k}D_{k}\left[D_{k}\right]_{+}S_{k}^{\T}$
  converge to $O$ when $k \to \infty$.
  The first term clearly converges to $O$ because $x^{*}$ is a feasible point of
  (\ref{eq:nsdp}), namely, $X(x^{*}) \succeq O$ and
  Lemma \ref{lemma:cardinality} $(ii)$--$(iii)$ hold.
  To show that the same can be said for the second term, observe that
  \begin{eqnarray*}
    \frac{1}{\sigma_{k-1}}S_{k}D_{k}\left[D_{k}\right]_{+}S_{k}^{\T}
    = \frac{1}{\sigma_{k-1}}\sum_{i=1}^{m}\lambda_{i}(\sigma_{k-1}Z_{k-1} - X(x_k))
    \left[\lambda_{i}(\sigma_{k-1}Z_{k-1} - X(x_k))\right]_{+}s_{i}^{k}\left[s_{i}^{k}\right]^{\T},
  \end{eqnarray*}
  where $s_{i}^{k}$ denotes the $i$-th column of $S_k$.
  From Lemma \ref{lemma:convergence}~($i$), we have
  \[
  \frac{1}{\sigma_{k-1}}\lambda_{i}(\sigma_{k-1}Z_{k-1} - X(x_k))
  \left[\lambda_{i}(\sigma_{k-1}Z_{k-1} - X(x_k))\right]_{+} \to 0,\ i = 1, \dots, m,
  \]
  which proves our claim. Therefore, we conclude that $x^{*}$ is a CAKKT point.\\

  \noindent \textbf{Case (b2)} Finally, we show that condition
  ($ii$) holds if $x^{*}$ is an infeasible point of (\ref{eq:nsdp}).
  Since $x_k \to x^{*}$ when $k \to \infty$, there exists $\delta > 0$
  and $\overline{k} \in \Re$ such that $x_{k} \in B(x^{*}, \delta)$
  for all $k \geq \overline{k}$. In the following, let $k \geq
  \overline{k}$. From the definition of $P$ given
  in~\eqref{eq:feasib}, we have
  \[
  \grad P(x) = \grad g(x)g(x) - \adjoint{x}\left[-X(x)\right]_{+}.
  \]
  From Lemma~\ref{lemma:convergence}~($ii$), there exists a positive
  constant $M$ such that $\norm{\grad P(x_k)/\sigma_{k-1}} \leq M$. Also, by
  Lemma \ref{lemma:cardinality} ($iii$), we have $\sigma_{k} \to 0$
  and $\norm{\grad P(x_k)} \to \norm{\grad P(x^{*})} = 0$ when $k \to
  \infty$. Therefore, condition ($ii$) holds if $x^{*}$ is an
  infeasible point of (\ref{eq:nsdp}).
\end{proof}


\section{Numerical experiments}
\label{sec:experiments}

In this section, we present some simple numerical experiments to check
the validity of Algorithm~\ref{algorithm:proposed}. We consider two
NSDP problems here: one with no KKT solution and another where
degeneracy occurs. The program was implemented in MATLAB R2020b, and
we ran Algorithm~\ref{algorithm:proposed} on a machine with Intel Core
i$9$-$9900$K, with $3.60$GHz of CPU and $128$GB of RAM. The subproblem
used in the algorithm is computed using SDPT3, version
4.0~\cite{TTT99,TTT03}.

Let us first describe the setting of
Algorithm~\ref{algorithm:proposed}.  The initial point was set as
$(x_{0}, y_{0}, Z_{0}) = (0,0,O)$.  For the stopping criteria, we use
the following three conditions:
\begin{eqnarray*}
  \begin{array}{c}
    r(x_k, y_k, Z_k) := r_{\mathrm{V}}(x_k) + r_{\mathrm{O}}(x_k, y_k, Z_k) \leq 10^{-4},
    \quad \gamma_k \leq 10^{-4}, \quad \mbox{or} \quad k_{\max} = 200.
  \end{array}
\end{eqnarray*}
As usual, the condition $\norm{\grad F(x_k; \sigma_k, y_k, Z_k)} = 0$
in Step~$2$ of Algorithm~\ref{algorithm:proposed} is relaxed to
$\norm{\grad F(x_k; \sigma_k, y_k, Z_k)} \leq 10^{-4}$. Moreover, the
constants and parameters were set as follows: $\tau := 10^{-4}$,
$\omega := 10^{-4}$, $\beta := 0.5$, $\kappa := 10^{-5}$,
$y_{\mathrm{max}} := 10^6$, $z_{\mathrm{max}} := 10^6$, $\epsilon :=
10^{-4}$, $\phi_0 := 10^3$, $\psi_0 := 10^3$, $\gamma_0 := 10^{-1}$,
$\sigma_0 := 10^{-1}$. We also set the stopping condition's parameter
for the subproblem (``gaptol'' in SDPT3) as $10^{-10}$.


\subsection{Problem with no KKT solution}

We consider the following one-dimensional problem:
\begin{eqnarray*}
  \begin{array}{ll}
    \underset{x \in \Re}{\mathrm{minimize}} & 2x \\
    \mathrm{subject\ to\ } &
    \left(\begin{array}{cc}
      0 & -x \\
      -x & 1
    \end{array}\right) \succeq O.
  \end{array}
\end{eqnarray*}
The unique feasible point, which is also optimal for this problem is
$x = 0$.  However, it is not a KKT point, as proved
in~\cite[Example~1]{Andreani2018}. For this problem,
Algorithm~\ref{algorithm:proposed} was able to obtain the solution at
the $35$-th iteration and the value of the measure $r$ was equal to
$9.98 \times 10^{-5}$.  Although this is a toy problem, we can see
that the absence of KKT does not affect the proposed SQSDP method, at
least in this example.


\subsection{Degenerated problem}
Next, we solve the following degenerated problem:
\begin{eqnarray*}
  \begin{array}{ll}
    \underset{X \in \S^n}{\mathrm{minimize}} & \langle C, X \rangle \\
    \mathrm{such\ that\ } & X_{ii} = 1, \quad i = 1, \dots , n, \\
    & \inner{J}{X} = 0,\ X \succeq O,
  \end{array}
\end{eqnarray*}
where $J$ is an $n$-dimensional square matrix with all elements equal to~$1$,
and $C$ is an $n$-dimensional symmetric matrix whose elements satisfy $C_{ij} \in [-1, 1]$
for all $i, j = 1, \dots, n$. Since $e^{\T}Xe = \inner{J}{X} = 0$, there is no strictly feasible point
for this problem, which means that Slater's constraint qualification does not hold.
\par
Here, we generated $10$ instances of problems with $n = 5$ and $n = 10$, and solved them by using
Algorithm~\ref{algorithm:proposed}. Table~\ref{table:problem2} shows its computational results,
where ``Average iterations'' means the average iterations spent by Algorithm~\ref{algorithm:proposed}.
In addition, ``Average $r$'', ``Maximum $r$'' and ``Minimum $r$'' represent, respectively, the average,
maximum, and minimum values of the measure $r$ at the final iteration. As it is possible to observe in the table, in most problems, the proposed algorithm was able to find a solution.
\begin{table}[htb]
	\begin{center}
	\caption{Performance for the degenerated problem}
	\label{table:problem2}
	\begin{tabular}{|c|c|c|} \hline
		& $n = 5$ & $n = 10$ \\ \hline \hline
		Average iterations & $183.6$ & $166.9$	\\ \hline
		Average $r$ & $2.45 \times 10^{-3}$ & $7.01 \times 10^{-3}$ \\ \hline
		Maximum $r$ & $1.50 \times 10^{-2}$ & $6.67 \times 10^{-2}$ \\ \hline
		Minimum $r$ & $6.09 \times 10^{-5}$ & $8.95 \times 10^{-5}$ \\ \hline
	\end{tabular}
	\end{center}
\end{table}


\section{Conclusion}
\label{sec:conclusion}
In this paper, we proposed a revised sequential quadratic semidefinite
programming method for NSDP. The algorithm is based on a recent work
of Yamakawa and Okuno~\cite{YO20}, and the main difference is that now
we are able to generate CAKKT points. By adding a not strict
assumption, called generalized \L{}ojasiewicz inequality, we also
proved the global convergence of the method. Simple numerical
experiments were also done, showing the validity of the method. One
future work will be to analyze the convergence rate of the method.




\bibliographystyle{plain}
\bibliography{journal-titles,references}


\section*{Appendix A}

Here, we give a proof for Lemma~\ref{lemma:convergence}. Before that,
we state the following useful result concerning
eigenvalues~\cite{HJ91}. If $A, B \in \S^{m}$, then for all $i = 1,
\dots, m$, we have
\begin{equation}
  \label{eq:eigenvalue}
  \lambda_{1}(A) + \lambda_{i}(B) \leq
  \lambda_{i}(A+B) \leq
  \lambda_{m}(A) + \lambda_{i}(B).
\end{equation}

\noindent \textbf{Proof of Lemma~\ref{lemma:convergence} (i):}\\

\noindent For simplicity, let us define
\[
\tilde{\lambda}_i^k := \lambda_i \big( \sigma_{k-1} Z_{k-1} - X(x_k) \big).
\]
We will prove that $(1/\sigma_{k-1}) \tilde{\lambda}_i^k
[\tilde{\lambda}_i^k]_+ \to 0$ by considering two cases, when $i \in
\{1,\dots,m\}$ satisfies (a) $\lambda_{i}(-X(x^{*})) < 0$ or (b)
$\lambda_{i}(-X(x^{*})) = 0$. \\
  
\noindent \textbf{Case (a)} Let $i \in \{1,\dots,m \}$ be such that
$\lambda_{i}(-X(x^{*})) < 0$.\\

\noindent First, recalling that $\sigma_{k-1} > 0$, it follows from
the second inequality of~\eqref{eq:eigenvalue} that
\begin{equation}
  \label{eq:inequality_of_eigenvalue}
  \displaystyle{\frac{1}{\sigma_{k-1}}\tilde{\lambda}_i^k}
  \leq \displaystyle{\lambda_{m}(Z_{k-1})
    + \lambda_{i}\left(-\frac{1}{\sigma_{k-1}}X(x_k)\right)}.
\end{equation}
From the assumption that $\lambda_{i}(-X(x^{*})) < 0$, there exists a
positive integer $\overline{k}$ such that $\lambda_{i}(-X(x_k)) < 0$
for all $k \geq \overline{k}$.  In the following, we suppose that $k
\geq \overline{k}$.  We obtain from Lemma \ref{lemma:cardinality}
$(ii)$--$(iii)$ that $\{Z_{k-1}\}$ is bounded and $\sigma_{k-1} \to 0$ when
$k \to \infty$. Thus, we have
\[
\lambda_{m}(Z_{k-1}) + \lambda_{i}\left(-\frac{1}{\sigma_{k-1}}X(x_k)\right) \to -\infty.
\]
This, together with~\eqref{eq:inequality_of_eigenvalue} shows that
$(1/\sigma_{k-1})\tilde{\lambda}_i^k \to -\infty$, which yields
$(1/\sigma_{k-1})[\tilde{\lambda}_i^k]_+ = 0$ for sufficiently
large~$k$.  Moreover, $\tilde{\lambda_{i}^k}$ is bounded from
Assumption~\ref{assumption:function} (b) and
Lemma~\ref{lemma:cardinality} $(ii)$, and therefore $(1/\sigma_{k-1})
\tilde{\lambda}_i^k [\tilde{\lambda}_i^k]_+ \to 0$ in this case.\\

\noindent \textbf{Case (b)} Let $i \in \{1,\dots,m \}$ be such that
$\lambda_{i}(-X(x^{*})) = 0$.\\

\noindent Here, we consider once again two cases: (b1)\ there
exists a positive integer $\hat{k}$ such that $\lambda_{i}(-X(x_k))
\leq 0$ for all $k \geq \hat{k}$, and (b2)\ there exists a subset
$\mathcal{J} \subset \mathcal{M}$ such that
$\mathrm{card}(\mathcal{J}) = \infty$ and $\lambda_{i}(-X(x_k)) > 0$
for all $k \in \mathcal{J}$.\\

\noindent \noindent \textbf{Case (b1)} Assume that there exists a
positive integer $\hat{k}$ such that $\lambda_{i}(-X(x_k)) \leq 0$
for all $k \geq \hat{k}$.
To prove our claim, here we will show that
$(1/\sigma_{k-1})[\tilde{\lambda}_i^k]_+$ is bounded and
$\tilde{\lambda}_i^k \to 0$ when $k \to \infty$. In the following, let
$k \geq \hat{k}$.

Let us first show that $(1/\sigma_{k-1})[\tilde{\lambda}_i^k]_+$ is
bounded. Since $\sigma_{k-1} > 0$, we have
\begin{align*}
  0 \leq &
  \: (1/\sigma_{k-1})[\tilde{\lambda}_i^k]_+ \\
  = & \:
  \max \left\{0, \lambda_{i}\left(Z_{k-1} - \frac{1}{\sigma_{k-1}}X(x_k)\right)\right\} \\
  \leq & \:
  \max \left\{0, \lambda_{m}(Z_{k-1})
  + \lambda_{i}\left(- \frac{1}{\sigma_{k-1}}X(x_k)\right)\right\} \\
  \leq & \: \max \left\{0, \lambda_{m}(Z_{k-1})\right\}
  + \max \left\{0, \lambda_{i}\left(-\frac{1}{\sigma_{k-1}}X(x_k)\right)\right\},
\end{align*}
where the second inequality holds from the second inequality
of~\eqref{eq:eigenvalue}, and the last one follows from a property of
the $[\,\cdot\,]_+$ function. From Lemma
\ref{lemma:cardinality}~($ii$), the first term of the last expression
is bounded. Moreover, $\lambda_{i}\left(-X(x_k)/\sigma_{k-1}\right)
\le 0$ from assumption and because $\sigma_{k-1} > 0$, which means that
the second term of the last inequality is equal to $0$. Therefore,
$(1/\sigma_{k-1})[\tilde{\lambda}_i^k]_+$ is bounded as we claimed.
Now, from Lemma \ref{lemma:cardinality} $(ii)$--$(iii)$, we get
\[
\tilde{\lambda}_i^k = \lambda_{i}(\sigma_{k-1}Z_{k-1} - X(x_k)) \to
\lambda_{i}(-X(x^*)) = 0,
\]
and the proof is complete for this case.\\

\noindent \noindent \textbf{Case (b2)} Assume that there exists a
subset $\mathcal{J} \subset \mathcal{M}$ such that $\mathrm{card}(\mathcal{J})
= \infty$ and $\lambda_{i}(-X(x_k)) > 0$ for all $k \in \mathcal{J}$.
Now, take $k \in \mathcal{J}$. From the first inequality
of~\eqref{eq:eigenvalue}, we obtain
\[
\tilde{\lambda}_i^k = \lambda_{i}^k (\sigma_{k-1}Z_{k-1} - X(x_k)) \geq
\sigma_{k-1}\lambda_{1}(Z_{k-1}) + \lambda_{i}(-X(x_k)).
\]
By assumption, we know $\lambda_{i}(-X(x_k)) > 0$. Moreover, from Step
$3$ of Procedure~\ref{algorithm:procedure}, $Z_{k-1}$ is a positive
semidefinite matrix and thus $\lambda_{1}(Z_{k-1}) \geq 0$.
Therefore, the above inequality shows that $\tilde{\lambda}_i^k >
0$. Now, recalling that $\sigma_{k-1} > 0$ and
from~\eqref{eq:eigenvalue},
\begin{equation}
  \label{eq:inequality_of_eigenvalue2}
  \displaystyle{\lambda_{1}(Z_{k-1})
    + \frac{1}{\sigma_{k-1}} \lambda_{i}\left(-X(x_k)\right)}
  \leq \displaystyle{\frac{1}{\sigma_{k-1}}\tilde{\lambda}_i^k}
  \leq \displaystyle{\lambda_{m}(Z_{k-1})
    + \frac{1}{\sigma_{k-1}} \lambda_{i}\left(-X(x_k)\right)}
\end{equation}
holds. Multiplying it by the positive term $\tilde{\lambda}_i^k$, and
observing that $[\tilde{\lambda}_i^k]_+ = \tilde{\lambda}_i^k$, we have
\begin{align}
  \lambda_{1}(Z_{k-1}) \tilde{\lambda}_i^k
  + \frac{1}{\sigma_{k-1}} \lambda_{i}\left(-X(x_k)\right) \tilde{\lambda}_i^k 
  & \leq \frac{1}{\sigma_{k-1}} \tilde{\lambda}_i^k [\tilde{\lambda}_i^k]_+
  \nonumber \\
  & \leq \lambda_{m}(Z_{k-1})\tilde{\lambda}_i^k
  + \frac{1}{\sigma_{k-1}} \lambda_{i}\left(-X(x_k)\right) \tilde{\lambda}_i^k.
  \label{eq:case_b2_ineq}
\end{align}
In addition, by Lemma \ref{lemma:cardinality} ($ii$)--($iii$) and the
fact that $\lambda_{i}(-X(x_k)) \to \lambda_{i}(-X(x^{*})) = 0$ when
$\mathcal{J} \ni k \to \infty$, we obtain $\tilde{\lambda}_i^k \to
0$. Moreover, $\lambda_{1}(Z_{k-1})\tilde{\lambda}_i^k \to 0$ and
$\lambda_{m}(Z_{k-1})\tilde{\lambda}_i^k \to 0$ hold. These limits,
together with inequalities~\eqref{eq:case_b2_ineq}, show that
$(1/\sigma_{k-1})\tilde{\lambda}_i^k [\tilde{\lambda}_i^k]_+ \to 0$
holds when
\begin{equation}
  \label{eq:case_b2_essential1}
  \frac{1}{\sigma_{k-1}}\lambda_{i}\left(-X(x_k)\right)\tilde{\lambda}_i^k \to 0.
\end{equation}
  
To prove this limit, note that $\lambda_i(-X(x_k)) > 0$ holds by
assumption, and so we can multiply (\ref{eq:inequality_of_eigenvalue2})
by it to obtain
\begin{align*}
  \lambda_{1}(Z_{k-1})\lambda_{i}(-X(x_k))
  + \frac{1}{\sigma_{k-1}} \lambda_{i}\left(-X(x_k)\right)^2
  & \leq \frac{1}{\sigma_{k-1}} \lambda_{i}(-X(x_k)) \tilde{\lambda}_{i}^k\\
  & \leq \lambda_{m}(Z_{k-1})\lambda_{i}(-X(x_k))
  + \frac{1}{\sigma_{k-1}} \lambda_{i}\left(-X(x_k)\right)^2.
\end{align*}
Once again from Lemma \ref{lemma:cardinality} ($ii$) and the fact that
$\lambda_{i}(-X(x_k)) \to 0$ when $\mathcal{J} \ni k \to \infty$, show
that $\lambda_{1}(Z_{k-1})\lambda_{i}(-X(x_k)) \to 0$ and
$\lambda_{m}(Z_{k-1})\lambda_{i}(-X(x_k)) \to 0$. These limits,
together with the above inequalities show that in order to prove
\eqref{eq:case_b2_essential1}, it becomes sufficient to prove
\begin{equation}
  \label{eq:case_b2_essential2}
  \frac{1}{\sigma_{k-1}} \lambda_{i}(-X(x_k))^2 \to 0.
\end{equation}

Let us now prove~\eqref{eq:case_b2_essential2}. Since $x_k \to x^*$
when $k \to \infty$, there exists $\delta > 0$ and a positive integer
$\overline{k}$ such that $x_{k} \in B(x^{*}, \delta)$ for all $k \geq
\overline{k}$. In the following, we suppose $\mathcal{J} \ni k \geq
\overline{k}$.  Regarding $P$ defined in~\eqref{eq:feasib}, we have
$P(x^{*}) = 0$ because $x^{*}$ is a feasible point of (\ref{eq:nsdp}).
Thus, from Assumption~\ref{assumption:function}~(d),
$|P(x^k)-P(x^{*})| \leq \varphi(x^k)\norm{\grad P(x^k)}$ holds, and
therefore
\[
  0 \leq \frac{1}{\sigma_{k-1}}|P(x_k)|
  = \frac{1}{\sigma_{k-1}}|P(x_k)-P(x^{*})|
  \leq \frac{\varphi(x_k)}{\sigma_{k-1}}\norm{\grad P(x_k)}.
\]
From Lemma~\ref{lemma:convergence} ($ii$), $\left\|\grad
P(x_k)/\sigma_{k-1}\right\|$ is bounded. Moreover, since $\varphi(x_k)
\to \varphi(x^{*}) = 0$, the above inequality shows that
\[
\frac{1}{\sigma_{k-1}}|P(x_k)|
= \frac{1}{2\sigma_{k-1}}\left(\norm{g(x_k)}^2 +
\norm{\left[-X(x_k)\right]_{+}}_{\mathrm{F}}^2\right) \to 0.
\]
In addition, because $\norm{g(x_k)}$ and
$\norm{\left[-X(x_k)\right]_{+}}_{\mathrm{F}}$ are both nonnegative,
we get the limit
$(1/\sigma_{k-1})\norm{\left[-X(x_k)\right]_{+}}_{\mathrm{F}}^2 \to
0$. Since
\[
\frac{1}{\sigma_{k-1}}\sum^{m}_{j=1}\lambda_{j}(\left[-X(x_k)\right]_{+})^2
= \frac{1}{\sigma_{k-1}}\norm{\left[-X(x_k)\right]_{+}}_{\mathrm{F}}^2
\to 0
\]
holds, we conclude that
$\lambda_{j}(\left[-X(x_k)\right]_{+})^2/\sigma_{k-1} \to 0$ for $j =
1, \dots, m$. Therefore, because $\lambda_{i}(-X(x_k)) > 0$ from
assumption, \eqref{eq:case_b2_essential2} holds. We then conclude that
\eqref{eq:case_b2_essential1} is also true, which in turn shows the
lemma's claim.\qed \\

\noindent \textbf{Proof of Lemma~\ref{lemma:convergence} (ii):}\\

\noindent Let us prove that $\left\|\grad P(x_k)/\sigma_{k-1}\right\|$
is bounded when $k \in \mathcal{M}$. Fixing $k \in \mathcal{M}$,
clearly,
\[
\grad P(x_k) = \grad g(x_k)g(x_k) - \adjoint{x_k}\left[-X(x_k)\right]_{+}
\]
by the definition of~$P$ in~\eqref{eq:feasib}. Then, simple
calculations show that
\begin{align}
  \left\|\frac{\grad P(x_k)}{\sigma_{k-1}}\right\| \;
  = & \;\; \frac{1}{\sigma_{k-1}} \Big\| \grad g(x_k)g(x_k)
    - \adjoint{x_k} \Big( \left[\sigma_{k-1}Z_{k-1} - X(x_k)\right]_{+} 
    \nonumber \\
    & \;\;{} - \left[\sigma_{k-1}Z_{k-1} - X(x_k)\right]_{+}
    + \left[-X(x_k)\right]_{+} \Big) \Big\|
  \nonumber \\
  \leq & \;\; \frac{1}{\sigma_{k-1}} \big\| \grad g(x_k)g(x_k)
    - \adjoint{x_k}\left[\sigma_{k-1}Z_{k-1} - X(x_k)\right]_{+} \big\|
  \nonumber \\
  & \;\;{} + \frac{1}{\sigma_{k-1}} \big\| \adjoint{x_k} \big(
  \left[-X(x_k)\right]_{+} - \left[\sigma_{k-1}Z_{k-1} - X(x_k)\right]_{+}
  \big) \big\|.
  \label{eq:inequality_of_gradient_p}		
\end{align}
Moreover, the first term of the right-hand side of the above expression can be
written as follows:
\begin{align}
  & \;\; \bigg\| \frac{1}{\sigma_{k-1}} \Big( \grad g(x_k)g(x_k)
    - \adjoint{x_k}\left[\sigma_{k-1}Z_{k-1} - X(x_k)\right]_{+} \Big) \bigg\|
   \label{eq:inequality_to_proof_bounded_1} \\
  = & \;\; \left\| \big( \grad f(x_k) - \grad g(x_k)y_{k-1} \big)
  + \frac{1}{\sigma_{k-1}} \grad g(x_k)g(x_k) \right.
  \nonumber \\
  & \;\;{}   
  - \left. \adjoint{x_k}\left[ Z_{k-1} - \frac{X(x_k)}{\sigma_{k-1}} \right]_{+} 
  - \big(\grad f(x_k) - \grad g(x_k)y_{k-1} \big) \right\|
  \nonumber \\
  \leq & \;\; \norm{\grad F(x_k; \sigma_{k-1}, y_{k-1}, Z_{k-1})}
  + \norm{\grad f(x_k) - \grad g(x_k)y_{k-1}}
  \nonumber \\
  \leq & \;\; \gamma_{k-1} + \norm{\grad f(x_k) - \grad g(x_k)y_{k-1}},
  \nonumber
\end{align}
where the first inequality holds from~\eqref{eq:grad_merit_function},
and the second one follows because the iterate $k \in \mathcal{M}$
means $(k-1) \in \mathcal{K}_M$, so the condition in Step~3 of
Procedure~\ref{algorithm:procedure} is true. Once again from the
Step~3 of Procedure~\ref{algorithm:procedure}, we know that $y_{k-1}$
is bounded. With Assumption~\ref{assumption:function}~(b), this
implies the boundedness of $\norm{\grad f(x_k) - \grad
  g(x_k)y_{k-1}}$. Thus, the term
in~\eqref{eq:inequality_to_proof_bounded_1} is also bounded.

Let us now analyze the last term of the
expression~\eqref{eq:inequality_of_gradient_p}. Recall that
\[
\adjoint{x}U :=
\left[
  \begin{array}{c}
    \inner{A_{1}(x)}{U} \\
    \vdots \\
    \inner{A_{n}(x)}{U}
  \end{array}
\right] =
\left[
  \begin{array}{c}
    \mathrm{svec}(A_{1}(x))^\T \mathrm{svec}(U) \\
    \vdots \\
    \mathrm{svec}(A_{n}(x))^\T \mathrm{svec}(U)
  \end{array}
\right] =
\left[
  \begin{array}{c}
    \mathrm{svec}(A_{1}(x))^\T \\
    \vdots \\
    \mathrm{svec}(A_{n}(x))^\T 
  \end{array}
\right] \mathrm{svec}(U)
\]
for any $x \in \Re^n$ and $U \in \S^d$, where the second equality
holds from~\eqref{eq:svec}. Using the above formulation, we have:
\begin{align}
  & \;\; \frac{1}{\sigma_{k-1}} \big\| \adjoint{x_k} \big( \left[-X(x_k)\right]_{+}
  - \left[\sigma_{k-1}Z_{k-1} - X(x_k)\right]_{+} \big) \big\|
  \nonumber \\
  = & \;\; \frac{1}{\sigma_{k-1}}
  \left\|
  \left[
    \begin{array}{c}
      \mathrm{svec}(A_1(x_k))^{\T} \\
      \vdots \\
      \mathrm{svec}(A_n(x_k))^{\T}
    \end{array}
    \right]
  \mathrm{svec} \big(\left[-X(x_k)\right]_{+}
  - \left[\sigma_{k-1}Z_{k-1} - X(x_k)\right]_{+} \big)
  \right\|
  \nonumber \\
  \leq & \;\; \frac{1}{\sigma_{k-1}}
  \left\|
  \left[
    \begin{array}{c}
      \mathrm{svec}(A_1(x_k))^{\T} \\
      \vdots \\
      \mathrm{svec}(A_n(x_k))^{\T}
    \end{array}
    \right]
  \right\| \norm{\mathrm{svec}(\left[-X(x_k)\right]_{+}
    - \left[\sigma_{k-1}Z_{k-1} - X(x_k)\right]_{+})}.
  \label{eq:appendix_2nd_bound}
\end{align}
Now, from~\eqref{eq:svec}, the last norm is bounded as follows:
\begin{align*}
  & \;\; \norm{\mathrm{svec}(\left[-X(x_k)\right]_{+}
    - \left[\sigma_{k-1}Z_{k-1} - X(x_k)\right]_{+})} \\
  = & \;\; \norm{\left[-X(x_k)\right]_{+}
    - \left[\sigma_{k-1}Z_{k-1} - X(x_k)\right]_{+}}_{\mathrm{F}}
  \nonumber \\
  \leq & \;\; \norm{\sigma_{k-1}Z_{k-1}}_{\mathrm{F}}
  = \sigma_{k-1} \norm{Z_{k-1}}_{\mathrm{F}},
\end{align*}
where the inequality follows from the expansion property of
projections. The above inequality, together with Assumption
\ref{assumption:function} (b) and Lemma \ref{lemma:cardinality}
($ii$), shows that~\eqref{eq:appendix_2nd_bound} is also
bounded. Therefore, from (\ref{eq:inequality_of_gradient_p}) the claim
holds.\qed

\end{document}